\documentclass[12pt, a4paper]{article}

\usepackage{amssymb}
\usepackage{amsmath}
\usepackage{amsthm}
 \usepackage{dsfont}
 \usepackage{eufrak}

  \newtheorem{theorem}{Theorem}
 \newtheorem{definition}{Definition}
  \newtheorem{corollary}{Corollary}
        
    \newtheorem{proposition}{Proposition}

\DeclareMathOperator{\Div}{div}
\DeclareMathOperator{\Id}{Id}

\begin{document}

 \title{Forces for the Navier--Stokes equations and the Koch and Tataru theorem.}
\author{Pierre Gilles Lemari\'e-Rieusset\footnote{LaMME, Univ Evry, CNRS, Universit\'e Paris-Saclay, 91025, Evry, France; e-mail : pierregilles.lemarierieusset@univ-evry.fr}}
\date{}\maketitle

{\it In dedication to Olga Ladyzhenskaya's 100th birthday}
 
\begin{abstract}
We consider the  Cauchy problem for the incompressible Navier--Stokes equations on the whole space $\mathbb{R}^3$, with initial value $\vec u_0\in {\rm BMO}^{-1}$ (as in Koch and Tataru's theorem) and with force $\vec f=\Div \mathbb{F}$ where smallness of $\mathbb{F}$ ensures existence of a mild solution in absence of initial value.  We study  the interaction of the two solutions and discuss the existence of global solution for the complete problem (i.e. in presence of initial value and forcing term) under smallness assumptions. In particular, we discuss the interaction between Koch and Tataru solutions and Lei-Lin's solutions (in $L^2\mathcal{F}^{-1}L^1$)  or solutions in the multiplier space $\mathcal{M}(\dot H^{1/2,1}_{t,x}\mapsto L^2_{t,x})$. \end{abstract}
 


  In this paper, we consider global mild solutions of the Cauchy problem for the incompressible Navier--Stokes equations on the whole space $\mathbb{R}^3$. When looking for assumptions that respect the symmetries of the Navier--Stokes equations (with respect to spatial translation or to dilations), one is lead to consider the initial data to be in $BMO^{-1}$ (this is the famous Koch and Tataru theorem \cite{KOT}) but there is no natural choice for the forcing term. We are going to consider forces that are known to lead to global mild solutions (if they are small enough) in the absence of initial value, but the interaction between those forces and an initial  value in $BMO^{-1}$ or between forces in different functional spaces has not been discussed in the literature.
  
   \section{Navier--Stokes equations with a forcing term and the Koch and Tataru theorem}

  Let us give a short description of mild solutions in critical spaces for the Navier--Stokes equations. 
We shall look for minimal regularity assumptions for thhe solution $\vec u$. It is therefore  better to write the non-linear term $\vec u\cdot \vec\nabla\vec u $ in the Navier--Stokes equations as $\Div (\vec u\otimes\vec u)$ (the two vector fields are equal when $\vec u$ is a regular divergence free vector field). The Navier--Stokes equations we study are then
 \begin{equation}\left\{
 \begin{split} &\partial_t \vec u =\Delta\vec u-\vec \nabla p +\Div (\mathbb{F}-\vec u\otimes\vec u)
\\& \Div\vec u=0
\\& \vec u(0,.)=\vec u_0
\end{split}\right.\end{equation}
Taking the divergence of the first equation, we get
$$ \Delta p=(\vec\nabla\otimes\vec\nabla)\cdot  (\mathbb{F}-\vec u\otimes\vec u)$$ and thus
 \begin{equation}\label{pressure} \Delta\vec\nabla p=\vec\nabla\left((\vec\nabla\otimes\vec\nabla)\cdot  ( \mathbb{F}-\vec u\otimes\vec u)\right) .\end{equation}  Assuming that $\vec\nabla p$ is equal to $0$ at infinity, equation (\ref{pressure}) defines $p$ as a function of $\mathbb{F}$ and $\vec u$. More precisely, if $\mathbb{F}-\vec u\otimes\vec u$ is assumed to belong to $L^1_{\rm loc}((0,+\infty), L^1(\frac{dx}{1+\vert x\vert^4}))$, then \cite{FLR} shows that the solution $\vec\nabla p$ of equation (\ref{pressure})  which is equal to $0$ at infinity is given by the formula $$\vec\nabla p=\frac 1 \Delta \vec\nabla\left((\vec\nabla\otimes\vec\nabla)\cdot  ( \mathbb{F}-\vec u\otimes\vec u)\right)$$ where, writing 
 $$ G(x)=\frac 1{4\pi\vert x\vert}$$ for the fundamental solution of $-\Delta$ $$(-\Delta G=\delta\text{ so that, for }f\in\mathcal{D}(\mathbb{R}^3),  f=G*(-\Delta f))$$ and choosing a function $\psi\in\mathcal{D}(\mathbb{R}^3)$ which is equal to $1$ on a neighbourhood of $0$, $ \frac 1 \Delta \partial_i\partial_j\partial_k $ is defined as
  $$ \frac 1 \Delta \partial_i\partial_j\partial_k f= -   \partial_i\partial_j\partial_k \left((\psi G)*f\right)  -   \left(\partial_i\partial_j\partial_k ((1-\psi) G)\right)*f. $$
    Thus, we have an equation with one unknown $\vec u$ and two data $\vec u_0$ and $\mathbb{F}$.
 
 Defining (formally) the Leray projection operator $\mathbb{P}$ as
 $$ \mathbb{P}=\Id -\frac 1\Delta \vec\nabla \Div,$$ the Navier--Stokes equations then become
  \begin{equation}\left\{
 \begin{split} &\partial_t \vec u =\Delta\vec u+\mathbb{P}\Div (\mathbb{F}-\vec u\otimes\vec u)
\\& \vec u(0,.)=\vec u_0, \quad \Div\vec u_0=0
\end{split}\right.\end{equation}
 This is viewed as a non-linear heat equation and is transformed into the Duhamel formula
 \begin{equation}\label{oseen} \vec u= e^{t\Delta} \vec u_0+\int_0^t e^{(t-s)\Delta} \mathbb{P}\Div(\mathbb{F}-\vec u\otimes\vec u)\, ds\end{equation} where $e^{t\Delta}$ is the  convolution operator with the heat kernel:   $ e^{t\Delta}f=W_t*f$ with
 $$ W_t=\frac 1{(4\pi t)^{3/2}} e^{-\frac{x^2}{4t}}.$$
 
 We rewrite (\ref{oseen}) as
 \begin{equation}\label{oseenbis} \vec u= e^{t\Delta} \vec u_0+\mathcal{L}(\mathbb{F})-  B(\vec u,\vec v) \end{equation} where $$ \mathcal{L}(\mathbb F)=\int_0^t e^{(t-s)\Delta}\mathbb{P}\Div\mathbb{F}\, ds$$ and $$B(\vec u,\vec v)=\int_0^t e^{(t-s)\Delta}\mathbb{P}\Div(\vec u\otimes\vec v)\, ds.$$ 
 [Remark that we defined $\mathbb{P}\Div\mathbb{F}$ for regular enough tensors $\mathbb{F}$, i.e. $\mathbb{F}\in L^1_{\rm loc}((0,+\infty), L^1(\frac{dx}{1+\vert x\vert^4}))$, but we may also consider more singular data $\mathbb{F}$, as long as we are able to give a sense to $\mathcal{L}(\mathbb{F})$.]

 To have lighter computations and notations, it is better to forget the vectorial setting of the problem and to look at the bilinear operator $B$ as a family of scalar operators acting on scalar functions: if $\vec w=B(\vec u,\vec v)$ then $w_i=\sum_{1\leq j,k\leq 3} B_{i,j,k}(u_j,v_k)$ with
 $$ B_{i,j,k}(u,v)= \int_0^t e^{t-s)\Delta} (\delta_{i,k}\partial_j-\frac 1{\Delta} \partial_i\partial_j\partial_k)(uv)\, ds.$$
 More generally, we define $\mathfrak{S}_1$ as the space of smooth functions $\sigma$  on $\mathbb{R}^3$ which are positively homogeneous of degree $1$ ($\sigma(\lambda\xi)=\lambda\sigma(\xi)$ for $\lambda>0$). To $\sigma\in \mathfrak{S}_1$, we associate the Fourier multiplier $\sigma(D)$ and the operator
  $$ B_\sigma(u,v)= \int_0^t e^{t-s)\Delta}  \sigma(D)(uv)\, ds.$$
 
 The formalism of global mild solutions of the Cauchy problem for the Navier--Stokes equations is then described by the following  definition and proposition:
 \begin{definition}$\ $\\
 An \textbf{adapted  }  Banach space  is a  Banach space $\mathcal{Y}$  of   locally   integrable  functions on $(0,+\infty)\times\mathbb{R}^3$   such   that, for  every $\sigma\in \mathfrak{S}_1$, the bilinear operator $B_\sigma$  is   bounded on $\mathcal{Y}$: 
 $$ \|  B_\sigma(  u,  v))\|_ \mathcal{Y}\leq C_\sigma     \|  u\|_\mathcal{Y} \|  v\|_\mathcal{Y}.$$  
 \end{definition}

 For vector fields   $\vec u$ with coordinates in $\mathcal{Y}$, we shall write $\vec u\in \mathcal{Y}$ instead of $\vec u\in \mathcal{Y}^3$. The following theorem  is then easy to check  (through the Banach contraction principle):
 
 \begin{proposition}$\ $\\
 Let 
 $\mathcal{Y}$  be an  adapted  Banach space. If $\vec u_0$ is a (divergence free) vector fields of tempered distribution on $\mathbb{R}^3$ such $e^{t\Delta}\vec u_0 \in \mathcal{Y}$, $\mathcal{L}(\mathbb{F})\in\mathcal{Y}$ and if $\vec u_0$  and $\mathbb{F}$ are small enough:
 $$ 4 C_0( \| e^{t\Delta}\vec u_0\|_\mathcal{Y}+  \|\mathcal{L}(\mathbb{F})\|_\mathcal{Y})<1,$$ where $C_0$ is the norm of the  bilinear operator $B$ on $\mathcal{Y}^3$,  then the Navier--Stokes problem (\ref{oseen}) has a global solution $\vec u$ with
 $$ \|\vec u\|_\mathcal{Y}\leq 2(  \| e^{t\Delta}\vec u_0\|_\mathcal{Y}+  \|\mathcal{L}(\mathbb{F})\|_\mathcal{Y}).$$
 \end{proposition} 

  Navier--Stokes equations have symmetries. In particular, we have the two following properties: if $\vec u$ is a solution of the Navier--Stokes problem (\ref{oseen}) with data $\vec u_0$ and $\mathbb{F}$, then
 \begin{itemize}
 \item[$\bullet$] [space translation] $\vec u(t,x-x_0)$ is a solution of the Navier--Stokes problem (\ref{oseen}) with data $\vec u_0(x-x_0)$ and $\mathbb{F}(t,x-x_0)$,
 \item[$\bullet$] [space dilation]  if $\lambda>0$, $\lambda\vec u(\lambda^2 t,\lambda x)$ is a solution of the Navier--Stokes problem (\ref{oseen}) with data $\lambda\vec u_0(\lambda x)$ and $\lambda^2\mathbb{F}(\lambda^2 t, \lambda x)$.
\end{itemize} 

In particular, we shall look for \textbf{critical spaces}  $  \mathcal{Y}$, meaning that we have invariance of the norms under space translations and space dilations: for every $x_0\in\mathbb{R}^3$ and $\lambda>0$
\begin{equation*} 
\|  u(t,x-x_0)\|_\mathcal{Y}=\|  u\|_\mathcal{Y},\phantom{blablabla} \|\lambda  u(\lambda^2t,\lambda x)\|_\mathcal{Y}=\|  u\|_\mathcal{Y}.
 \end{equation*}
Finally, in order to give sense to the formula
$$ \vec u(t,.)=\vec u_0+\Delta \int_0^t \vec u(s,.)\, ds+\mathbb{P}\Div(\int_0^t \mathbb{F}-\vec u\otimes\vec u\, ds),$$
we require the continuous embedding   $\displaystyle \mathcal{Y}\subset \bigcap_{T>0} L^2((0,T),L^2(\frac{dx}{1+\vert x\vert^4}))$  [Due to the invariance through space translations or space dilations, it is equivalent to ask that $  u\mapsto\mathds{1}_{(0,1)\times B(0,1)}  u $ is bounded from $\mathcal{Y}$ to $L^2((0,1)\times B(0,1))$.]

In particular, we have $$ \mathcal{Y}\subset  \mathcal{Y}_2=\{  u\ /\ \sup_{t>0, x_0\in \mathbb{R}^3}  t^{-3/4} \|  u\|_{L^2((0,t)\times B(x_0,\sqrt t))} <+\infty\}.$$ Thus, $ \mathcal{Y}_2$ is maximal in the class of Banach spaces $\mathcal{Y}$ that satisfy the conditions $\|  u(t,x-x_0)\|_\mathcal{Y}=\|  u\|_\mathcal{Y}$, $\|\lambda  u(\lambda^2t,\lambda x)\|_\mathcal{Y}=\|  u\|_\mathcal{Y}$ and $$\sup_{\|  u\|_\mathcal{Y}\leq 1} \int_0^1\int_{B(0,1)} \vert  u(s,y)\vert^2\, ds\, dy<+\infty.$$

Koch and Tataru \cite{KOT}  identified the space $\mathcal{X}$ such that $  u_0\in \mathcal{X}$ implies $S(  u_0)\in \mathcal{Y}_2$: 

\begin{proposition}$\ $\\  For a tempered distribution $  u_0$, the following assertions are equivalent:
\begin{itemize}
\item[(i)] $e^{t\Delta}  u_0\in \mathcal{Y}_2$;
\item[(ii)]                 $ u_0\in {\rm BMO}^{-1}=\dot F^{-1}_{2,\infty}$ (i.e., there exists $ \vec  v_0$ in ${\rm BMO}^3$ such that $  u_0=\Div\vec v_0$).
\end{itemize} 
\end{proposition}

However, $  \mathcal{Y}_2$ is not an adapted  Banach space  (see a counter-example in Proposition \ref{propy2} in the Appendix). 
 The Koch and Tataru theorem deals with a subspace of $\mathcal{Y}_2$. We define the  space 
$$  \mathcal{Z}_{0}=\{  u\in L^1_{\rm loc}((0,+\infty)\times\mathbb{R}^3\ /\  \sup_{t>0}\sqrt t  \|  u(t,.)\|_\infty<+\infty\}.$$ The Koch and Tataru space 
 $\mathcal{Y}_{KT}$ is then defined as:
$$  \mathcal{Y}_{KT}=  \mathcal{Y}_2\cap \mathcal{Z}_0 .$$   $\mathcal{Y}_{KT}$ is normed with $\|  u\|_{\mathcal{Y}_{KT}}= \|  u\|_{\mathcal{Y}_2}+\sup_{t>0} \sqrt t\|  u(t,.)\|_\infty$, where
$$  \|  u\|_{\mathcal{Y}_2}=\sup_{t>0, x_0\in \mathbb{R}^3}  t^{-\frac 3 {4}} \|   u\|_{L^2((0,t)\times B(x_0,\sqrt t))} .$$

Koch and Tataru's theorem is then the following one \cite{KOT, PGL1}:

\begin{theorem}\label{kocht}$\ $\\ A) For    every $\sigma\in \mathfrak{S}_1$, the bilinear operator $B_\sigma$  is      a bounded bilinear operator from  $\mathcal{Y}_{KT}\times\mathcal{Y}_2$  to  $\mathcal{Y}_2$. It is also a bounded bilinear operator from  $\mathcal{Y}_{KT}\times\mathcal{Y}_{KT}$  to  $\mathcal{Y}_{KT}$.
\\ B)   The following assertions are equivalent:
\begin{itemize}
\item[(i)] $e^{t\Delta}  u_0\in \mathcal{Y}_{KT}$;
\item[(ii)]                 $ u_0\in {\rm BMO}^{-1}$;
\end{itemize} 
C)
 $ \mathcal{Y}_{KT}$  is an  adapted  Banach space.
Thus, there exists a positive constant $\epsilon_0$ such that, if $\|e^{t\Delta}\vec u_0\|_{\mathcal{Y}_{KT}}+\| \mathcal{L}(\mathbb{F})\|_{\mathcal{Y}_{KT}}<\epsilon_0$, then the Navier--Stokes problem (\ref{oseen}) has a global mild solution $\vec u\in\mathcal{Y}_{KT}$.
\end{theorem}
 
 \begin{proof} We sketch the proof given by Koch and Tataru in \cite{KOT}, and try to highlight the obstructing term for proving the boundedness of $B_\sigma$ on $\mathcal{Y}_2$.  A simple but key estimate is the following control:
 $$      \vert e^{(t-s)\Delta}\sigma(D)(u(s,.)v(s,.))\vert\leq C_\sigma  \int \frac 1{(\sqrt{t-s}+\vert x-y\vert)^4} \vert u(s,y)\vert\, \vert v(s,y)\vert\, dy.
 $$                                 
  We need to estimate, for every $T>0$ and $x\in\mathbb{R}^3$, 
$\| \mathds{1}_{Q_{T,x}}  B_\sigma(u,v)\|_{L^2L^2}$ where
$$Q_{T,x}=\{(t,y)\ /\ 0<t<T, \vert x-y\vert\leq \sqrt T\} .$$
Koch and Tataru split $w=B_\sigma(u,v)$ in three parts: 
\begin{itemize}
\item[$\bullet$] $w_1=B_\sigma(u,(1-\mathds{1}_{Q_{10 T,x}} )v)$ : we  easily check that $ \mathds{1}_{Q_{ T,x}} \vert w_1\vert\leq C \frac 1 {\sqrt T} \|u\|_{\mathcal{Y}_2} \|v\|_{\mathcal{Y}_2}$ and thus $$\|  \mathds{1}_{Q_{T,x}}  w_1\|_{L^2L^2}\leq C T^{3/4} \|u\|_{\mathcal{Y}_2} \|v\|_{\mathcal{Y}_2}.$$
\item[$\bullet$] $w_2(t,y)= \sigma(D) e^{t\Delta}\int_0^t \mathds{1}_{Q_{10 T,x}} uv\, ds$. The main lemma in Koch and Tataru's proof states that the operator
$ Q(u,v)= \sqrt{-\Delta}e^{t\Delta}\int_0^t \mathds{1}_{Q_{10 T,x}} uv\, ds$ maps $\mathcal{Y}_2\times \mathcal{Y}_2$ to $L^2 L^2$ with a norm of order $T^{3/4}$ and thus
$$\|  \mathds{1}_{Q_{T,x} } w_2\|_{L^2L^2}\leq  \|w_2\|_{L^2 L^2} \leq C T^{3/4} \|u\|_{\mathcal{Y}_2} \|v\|_{\mathcal{Y}_2}.$$
\item[$\bullet$] $w_3(t,y)= \sigma(D) \int_0^t  (e^{(t-s)\Delta}-e^{t\Delta}) \sqrt{-\Delta}(\mathds{1}_{Q_{10 T,x}} uv)\, ds$. They rewrite $w_3$ as
$$ w_3 =\frac{\sigma(D)}{\sqrt{-\Delta}}\int_0^t e^{(t-s)\Delta} \Delta \frac{e^{s\Delta}-\Id}{\sqrt{-\Delta}} (\mathds{1}_{Q_{10 T,x}} uv)\, ds$$ and use the maximal regularity of the heat kernel in $L^2L^2$ to write
$$ \|w_3\|_{L^2L^2}\leq C\| \frac{e^{t\Delta}-\Id}{\sqrt{-\Delta}} (\mathds{1}_{Q_{10 T,x}} uv)\|_{L^2L^2}\leq C' \|\sqrt t  \mathds{1}_{Q_{10 T,x}} uv\|_{L^2L^2}.$$ Thus,
$$\|  \mathds{1}_{Q_{T,x}}  w_3\|_{L^2L^2}\leq  \|w_3\|_{L^2 L^2} \leq C \|\sqrt t u\|_\infty \|\mathds{1}_{Q_{10 T,x}} v\|_{L^2L^2}\leq C T^{3/4} \|u\|_{\mathcal{Y}_{KT}} \|v\|_{\mathcal{Y}_2}.$$
\end{itemize} \textbf{Thus, the obstruction for the boundedness of $B_\sigma$ on $Y_2$ lies in $w_3$.}

 To finish the proof, we need to  establish the  control of $B_\sigma$ in $L^\infty$ norm. Writing 
 $$  \vert B_\sigma(u,v)(t,x)\vert\leq C_\sigma \int_0^t \int \frac 1{(\sqrt{t-s}+\vert x-y\vert)^4} \vert u(s,y)\vert\, \vert v(s,y)\vert\, dy\, ds,$$ we check that
$$ \int_0^{t/2} \int \frac 1{(\sqrt{t-s}+\vert x-y\vert)^4} \vert u(s,y)\vert\, \vert v(s,y)\vert\, dy\, ds\leq C \frac 1{\sqrt{t}} \|u\|_{\mathcal{Y}_2} \|v\|_{\mathcal{Y}_2}$$
 and
\begin{equation*}\!\int_{t/2}^t \!\int \!\frac { \vert u(s,y)\vert\, \vert v(s,y)\vert}{(\sqrt{t-s}+\vert x-y\vert)^4}\, dy\, ds\leq  \frac C{\sqrt{t}} \sup_{s>0}\sqrt s \|u(s,.)\|_\infty \sup_{s>0} \sqrt s \|v(s,.)\|_\infty.
{\qedhere} \end{equation*}
\end{proof}

The proof in Theorem \ref{kocht} is assumed to satisfy $\mathcal{L}(\mathbb{F})\in{\mathcal{Y}_{KT}}$, but one may   consider another forcing term; We have the obvious result:

\begin{proposition}\label{propclas}$\ $  \\ Let $ \mathcal{Y}$  be an  adapted  Banach space such that, for every    $\sigma\in \mathfrak{S}_1$, the bilinear operator $B_\sigma$  is a bounded bilinear operator from  $\mathcal{Y} \times L^{2,\infty}L^\infty$  to  $\mathcal{Y}$.  
Then, there exists a positive constant $\epsilon_0$ such that, if $\| \vec u_0\|_{{\rm BMO}^{-1}}+\| \mathcal{L}(\mathbb{F})\|_{\mathcal{Y}}<\epsilon_0$, then the Navier--Stokes problem (\ref{oseen}) has a global mild solution $\vec u\in\mathcal{Y}_{KT}+\mathcal{Y}$.
 \end{proposition}
 
 Let us notice that many adapted spaces studied in the literature satisfy the assumption of Proposition \ref{propclas}. Here are some examples:
 \begin{enumerate}
 \item[a)]  the Serrin class $\mathcal{Y}=L^{p}((0,+\infty), L^{q}(\mathbb{R}^3))$ with $2<p<+\infty$ and $\frac 2 p + \frac 3 q=1$   [this corresponds to the solutions of Fabes, Jones and Rivi\`ere [Fab72]);
 \item[b)]  direct generalizations of the Serrin class such that $\mathcal{Y}=L^{p,\rho}((0,+\infty)$, $\mathcal{Y}=L^{q,\sigma}(\mathbb{R}^3))$  or $\mathcal{Y}=L^{p,s}((0,+\infty), \dot M^{r,q}(\mathbb{R}^3))$ with $2<p<+\infty$, $\frac 2 p + \frac 3 q=1$, $1\leq \rho,\sigma\leq +\infty$ and $1<r\leq q$;
 \item[c)] the time-weighted Serrin class:  
$$\mathcal{Y}=\{u\ /\ t^{\alpha} u\in L^{p,\rho}((0,+\infty), L^{q,\sigma}(\mathbb{R}^3))\}$$
  with $3<q<+\infty$,  $2<p\leq+\infty$, $0\leq \alpha$, $1\leq \rho,\sigma\leq +\infty$ and $2\alpha+\frac 2 p + \frac 3 q=1$ [if $p=+\infty$, $L^{p,\rho}$ is to be replaced with $L^\infty$] (this corresponds to the solutions  considered by  Cannone and Planchon  \cite{CPL} or Kozono and Yamazaki \cite{KOY} and more recently by Farwig, Giga and Shu \cite{FGS} and Kozono and Shimizu \cite{KOS});
  \item[d)] the case $\mathcal{Y}=L^{\infty}((0,+\infty), L^{3,\infty} (\mathbb{R}^3))$, which is the endpoint of the Serrin class $L^p L^q$ with $p=+\infty$ and corresponds  to  the solutions of Kozono \cite{KOZ}  and Meyer \cite{MEY};
    \item[e)] the case of the Lorentz space $L^{5,\rho}_{t,x}=L^{5,\rho}((0,+\infty)\times\mathbb{R}^3)$ with $1\leq \rho\leq +\infty$  seems to be new but is easy:  to check that $L^{5,\rho}_{t,x}$ is an adapted Banach space, just notice that $\frac 1{(\sqrt{\vert t\vert}+\vert x\vert)^4}\in L^{5/4,\infty}_{t,x}$ and use convolution inequalities in Lorentz spaces; to check that $B_\sigma$ is bounded from  $L^{5,\rho}_{t,x} \times L^{2,\infty}L^\infty$  to  $L^{5,\rho}_{t,x}$, just notice that it is bounded from $L^p_tL^p_x \times L^{2,\infty}L^{\infty}$  to  $L^p_tL^p_x$ for $2<p<+\infty$ and conclude by interpolation.
 \end{enumerate}
 
 All those examples are embedded into larger classes of adapted spaces, namely the parabolic Morrey spaces 
  $\dot{ \mathcal{M}}_2^{p,5}$ where $2<p\leq 5$:
  $$\sup_{r>0, t\in\mathbb{R},x\in\mathbb{R}^3}   \frac 1 {r^{\frac 5 {p}- 1}}\left(\iint_{(t-r^2,t+r^2)\times B(x, r), s>0} \vert u(s,y)\vert^p\, dy\, ds\right)^{1/p}<+\infty.$$ 
  We have $L^p L^q\subset \dot{ \mathcal{M}}_2^{\min(p,q),5}$ [case a)], $L^{p,s}((0,+\infty), \dot M^{r,q}(\mathbb{R}^3))\subset  \dot{ \mathcal{M}}_2^{\sigma,5}$  with $2<\sigma<\min(p,r) $ [case b)], $ t^{\alpha} u\in L^{p,\rho}((0,+\infty), L^{q,\sigma}(\mathbb{R}^3)) \implies u\in  \dot{ \mathcal{M}}_2^{\delta,5}$  with $2<\delta<\min{ \frac{1}{1 +2\alpha p}p, q)} $ [case c)], $L^{\infty}L^{3,\infty} \subset \dot{ \mathcal{M}}_2^{r,5}$  for $2<r<3$ [case d)],  $L^{5,\rho}_{t,x} \subset \dot{ \mathcal{M}}_2^{r,5}$  for $2<r<5$ [case e)].
  We have the easy result on parabolic Morrey spaces:
  
  \begin{proposition}\label{propmorrey}$\ $  \\   For $2<p\leq 5$,  the parabolic Morrey spaces 
  $\dot{ \mathcal{M}}_2^{p,5}$  is an  adapted  Banach space and, for every    $\sigma\in \mathfrak{S}_1$, the bilinear operator $B_\sigma$  is a bounded bilinear operator from  $\dot{ \mathcal{M}}_2^{p,5}\times L^{2,\infty}L^\infty$  to  $\dot{ \mathcal{M}}_2^{p,5}$.  
Thus, there exists a positive constant $\epsilon_0$ such that, if $\| \vec u_0\|_{{\rm BMO}^{-1}}+\| \mathcal{L}(\mathbb{F})\|_{\dot{ \mathcal{M}}_2^{p,5}}<\epsilon_0$, then the Navier--Stokes problem (\ref{oseen}) has a global mild solution $\vec u\in\mathcal{Y}_{KT}+\dot{ \mathcal{M}}_2^{p,5}$.
 \end{proposition}
 
 \begin{proof} The fact that $\dot{ \mathcal{M}}_2^{p,5}$  is an  adapted  Banach space is proved in \cite{PGL2, PGL3}. One writes that, for  $\sigma\in \mathfrak{S}_1$,
  $$      \vert B_\sigma(u,v))\vert\leq C_\sigma  \int_0^t \int \frac 1{(\sqrt{t-s}+\vert x-y\vert)^4} \vert u(s,y)\vert\, \vert v(s,y)\vert\, dy\, ds.
 $$    Thus, $B_\sigma(u,v)$ is controlled by the parabolic Riesz potential of $\vert uv\vert$;  as  $\vert uv\vert\in \dot{ \mathcal{M}}_2^{p/2,5/2}$ if $u$ and $v$ belong to $\dot{ \mathcal{M}}_2^{p,5}$, we conclude by Hedberg's inequality for Riesz potentals and Morrey spaces that $B_\sigma(u,v)$ is controlled in $\dot{ \mathcal{M}}_2^{p,5}$.
 
 Now, let us consider $u\in \dot{ \mathcal{M}}_2^{p,5}$ and $v\in L^{2,\infty}L^\infty$. Since $p>2$, we   have $v\in \dot M^{\frac p{p-1},5}_2$, hence $uv\in \dot  { \mathcal{M}}_2^{1,5/2}$.  For $r>0$, $t\in\mathbb{R}$ and $x\in\mathbb{R}^3$, we want to estimate the $L^p L^p$ norm of  $B_\sigma(u,v)$ on $Q_r(t,x)=(t-r^2,t+r^2)\times B(x,r)$. Let $\rho(t-s,x-y)=\sqrt{t-s}+\vert x-y\vert$ be the parabolic distance. Let $(s,y)\in Q_r(t,x)$ and $(\sigma,z)$ be such that $8\ 2^k r\leq \rho(t-\sigma, x-z)\leq 16\ 2^k r$ with $k\in\mathbb{N}$; then $$\frac 1{(\sqrt{\sigma-s}+\vert z-y\vert)^4} \leq C \frac 1{(2^kr)^4}$$ and thus
\begin{equation*}\begin{split} \iint_{8\ 2^k r\leq \rho(t-\sigma, x-z)\leq 16\ 2^k r}& \frac 1{(\sqrt{\sigma-s}+\vert z-y\vert)^4}  \vert u(\sigma,z) v(\sigma,z)\vert\, d\sigma\, dz\\ \leq & C  \frac 1{(2^kr)^4} \iint_{Q_{16\, 2^k r}(t,x)} \vert u(\sigma,z) v(\sigma,z)\vert\, d\sigma\, dz\\ \leq & C'  \frac 1{2^kr} \|uv\|_{ \dot  { \mathcal{M}}_2^{1,5/2}}
\end{split}\end{equation*} so that
\begin{equation*}\begin{split} \iint_{Q_r(t,x)}& \left\vert  \iint_{8  r\leq \rho(t-\sigma, x-z))} \frac 1{(\sqrt{\sigma-s}+\vert z-y\vert)^4}  \vert u(\sigma,z) v(\sigma,z)\vert\, d\sigma\, dz\right\vert^p\, ds\, dy\\ \leq & C \iint_{Q_r(t,x)}\left\vert \sum_{k=0}^{+\infty}  \frac 1{2^kr} \|uv\|_{ \dot  { \mathcal{M}}_2^{1,5/2}}\right\vert^p \, ds\, dy \\ \leq & C' r^{5-p} \|u\|_{ \dot{ \mathcal{M}}_2^{p,5}}^p \|v\|_{L^{2,\infty}L^\infty}^p.
\end{split}\end{equation*}
On the other hand, we have $$\mathds{1}_{Q_{8r}(t,x)} u\in L^p L^p$$ so that
$$w(s,y)=  \iint_{ \rho(t-\sigma, x-z)<8r} \frac 1{(\sqrt{\sigma-s}+\vert z-y\vert)^4}  \vert u(\sigma,z) v(\sigma,z)\vert$$ satisfies
$$ \|w\|_{L^p(dy)}\leq C \int \frac 1{\sqrt{\vert s-\sigma\vert}} \|\mathds{1}_{Q_{8r}(t,x)} u(s,z)\|_{L^p(dz)} \|v(s,z)\|_{L^\infty(dz)}\, ds $$ and 
$$ \|w\|_{L^pL^p}\leq C'  \|\mathds{1}_{Q_{8r}(t,x)} u\|_{L^p L^p} \|v\|_{L^{2,\infty}L^\infty}\leq C''  r^{1-\frac p 5}\|u\|_{\dot{ \mathcal{M}}_2^{p,5}} \|v\|_{L^{2,\infty}L^\infty}.$$
Thus, $B_\sigma(u,v)$ belongs to $\dot{ \mathcal{M}}_2^{p,5}$.
 \end{proof}
  
 However, for some adapted spaces $ \mathcal{Y}$, assumptions in Proposition \ref{propclas} are not satisfied and we cannot use Proposition \ref{propmorrey} as they are not included in $\dot{ \mathcal{M}}_2^{p,5}$ for any $p>2$. For instance, let us consider $\mathcal{Y}=L^{2}((0,+\infty), \mathrm{A}(\mathbb{R}^3))$ where $ \mathrm{A}$ is the inverse Fourier transform of $L^1$  (this corresponds to the endpoint $p=2$ of the Serrin class and to  the solutions of Lei and Lin \cite{LEI} in $L^2   \mathrm{A}$). Then, obviously, $L^2  \mathrm{A}$ is included in $\dot{ \mathcal{M}}_2^{2,5}$ but not in $\dot{ \mathcal{M}}_2^{p,5}$ for $p>2$. Moreover,  it is easy to check that, for $\sigma_0(\xi)=\vert\xi\vert$, $B_{\sigma_0}$ is not bounded from  $L^2  \mathrm{A} \times\ {L^{2,\infty}L^\infty}$  to  $L^2  \mathrm{A}$, and even from $L^2  \mathrm{A} \times\mathcal{Y}_{KT} $  to  $L^2  \mathrm{A}$ nor to $ \mathcal{Y}_{KT}$    (see  counter-examples in Proposition \ref{propy3} in the Appendix). 
 
 Another adapted space for which we don't know whether we may apply Proposition \ref{propclas} is the space of singular multipliers  $\mathcal{M}(\dot H^{1/2,1}_{t,x}\mapsto L^2_{t,x})$. This space has been introduced by Lemari\'e-Rieusset \cite{PGL2, PGL3} and independently by Dao and Nguyen \cite{DAO}. 
 Notice that, for $2<p\leq 5$, we have the embeddings
 $$ \dot{ \mathcal{M}}_2^{p,5}\subset \mathcal{M}(\dot H^{1/2,1}_{t,x}\mapsto L^2_{t,x}) \subset \dot{ \mathcal{M}}_2^{2,5}$$ so that $ \mathcal{M}(\dot H^{1/2,1}_{t,x}\mapsto L^2_{t,x})$ may be viewed as an endpoint of the scale of adapted spaces $\dot{ \mathcal{M}}_2^{p,5}$ with $p>2$. Remark that $L^2  \mathrm{A} $ is not included in $ \mathcal{M}(\dot H^{1/2,1}_{t,x}\mapsto L^2_{t,x})$  (see  a counter-example in Proposition \ref{propy4} in the Appendix).

 Thus, we need to find a new adapted space if we want to consider the Cauchy problem with an initial value in ${\rm BMO}^{-1}$ and a forcing term $\Div\mathbb{F}$ leading (in absence of initial value) to a solution in $L^2\mathrm{A}$ or in $\mathcal{M}(\dot H^{1/2,1}_{t,x}\mapsto L^2_{t,x})$. This will be done in the next section by modifying the space $ \mathcal{Y}_{KT}$ of Koch and Tataru.

 \section{A variation on the Koch and Tataru theorem}
  Recall that the Koch and Tataru space $ \mathcal{Y}_{KT}$ is defined as 
  $$   \mathcal{Y}_{KT}=\{u \in \mathcal{Y}_2\ /\ \sup_{t>0} \sqrt t \|u(t,.)\|_\infty<+\infty\}.$$ It has been designed to grant that,  for every $\sigma\in \mathfrak{S}_1$, the bilinear operator $B_\sigma$  is   a bounded bilinear operator from  $\mathcal{Y}_{KT}\times\mathcal{Y}_2$  to  $\mathcal{Y}_2$.  $B_\sigma$  is also a bounded bilinear operator from  $\mathcal{Y}_{KT}\times\mathcal{Y}_{KT}$  to  $\mathcal{Y}_{KT}$.
  
  Let us remark that we may easily check that $\mathcal{Y}_{KT}\subset \dot{ \mathcal{M}}_2^{2,5}\subset \mathcal{Y}_2$. We are going to describe new spaces $\mathcal{Y}_{KT,q}$ with $5<q<+\infty$ so that
  $$ \mathcal{Y}_{KT}\subset \mathcal{Y}_{KT,q} \subset\dot{ \mathcal{M}}_2^{2,5}$$ and, for  every $\sigma\in \mathfrak{S}_1$, the bilinear operator $B_\sigma$  is  a bounded bilinear operator from  $\mathcal{Y}_{KT,q}\times \dot{ \mathcal{M}}_2^{2,5}$ to  $\mathcal{Y}_{KT,q}$.  In particular,  $\mathcal{Y}_{KT,q}$  is an adapted space and, for every adapted space $\mathcal{Y}$ such that $\mathcal{Y}\subset\dot{ \mathcal{M}}_2^{2,5}$, $\mathcal{Y}_{KT,q}+\mathcal{Y}$ is an adapted space.
  
  Recall that $$Q_{T,x}=\{(t,y)\ /\ 0<t<T, \vert x-y\vert\leq \sqrt T\} $$ and define $$R_{T,x}=\{(t,y)\ /\ T/2<t<T, \vert x-y\vert\leq \sqrt T\}. $$  If $u\in  \mathcal{Y}_{KT}$, then $\mathds{1}_{Q_{T,x}} u\in L^2_{t,x}$ and $\|\mathds{1}_{Q_{T,x}} u\|_{L^2_{t,x}}\leq   \|u\|_{\mathcal{Y}_2} T^{3/4}$. Moreover, $\mathds{1}_{R_{T,x}} u\in L^\infty_{t,x}$ and $\|\mathds{1}_{R_{T,x}} u\|_{L^\infty_{t,x}}\leq   \|u\|_{\mathcal{Z}_0} \frac {\sqrt 2}{\sqrt T}$. Thus, we have, for $5\leq q\leq +\infty$, 
  $$  \|\mathds{1}_{R_{T,x}} u\|_{\dot {\mathcal{M}}^{\frac{2q}5,q}_2}\leq C  \|\mathds{1}_{R_{T,x}} u\|_{L^q_{t,x}} \leq C' \|u\|_{\mathcal{Y}_{KT}} T^{\frac 5 {2q}-\frac 1 2}.$$
  
  \begin{definition} $\ $\\ The modified Koch and Tataru space $\mathcal{Y}_{KT,q}$ for $5 < q<+\infty$ is defined as the space of functions $u$ on $(0,+\infty)\times\mathbb{R}^3$ such that
  $$ \sup_{T>0,x\in \mathbb{R}^3} T^{-3/4} \|\mathds{1}_{Q_{T,x}} u\|_{L^2_{t,x}}<+\infty$$
  and
    $$ \sup_{T>0,x\in \mathbb{R}^3}  T^{-\frac 5 {2q}+\frac 1 2}  \|\mathds{1}_{R_{T,x}} u\|_{\dot {\mathcal{M}}^{\frac{2q}5,q}_2}<+\infty$$
    where
    $$ \| f\|_{\dot {\mathcal{M}}^{\frac{2q}5,q}_2}= \sup_{r>0, t\in\mathbb{R}, x\in\mathbb{R}^3}  r^{- \frac {15}{2q}}   \left(\iint_{(t-r^2,t+r^2)\times B(x,r)} \vert f(s,y)\vert^{\frac{2q}5} \, ds\, dy\right)^{\frac 5{2q}}.$$
\end{definition}

We first remark that $\mathcal{Y}_{KT,q}\subset \dot{ \mathcal{M}}_2^{2,5}$: if we want to estimate the $L^2_{t,x}$ norm of $u$ on $(t-r^2,t+r^2)\times B(x,r)$, we may assume that $t\geq r^2$ (otherwise, we  control the norm of $u$ on $(t-r^2,t+r^2)\times B(x,r)$ by the norm of $u$ on   $(-r^2,r^2)\times B(x,r)$); if $r^2\leq t\leq 4r^2$, we have $ (t-r^2,t+r^2)\times B(x,r) \subset Q_{5r^2,x}$, so we have a control of the $L^2_{t,x}$ norm of $u$ on $(t-r^2,t+r^2)\times B(x,r)$  by $\|u\|_{\mathcal{Y}_2} r^{3/2}$; if $t>4 r^2$, we have $ (t-r^2,t+r^2)\times B(x,r) \subset R_{\frac 3 2 t,x}$, so we have a control of the $L^2_{t,x}$ norm of $u$ on $(t-r^2,t+r^2)\times B(x,r)$  by $r^{5(\frac 12-\frac 1 q)} \|\mathds{1}_{R_{\frac 3 2 t,x}} u\|_{\dot {\mathcal{M}}^{\frac{2q}5,q}_2}$, hence in $r^{3/2} \left(\frac {r^2}t\right)^{ \frac 1 2-\frac 5{2q} } t^{\frac 12-\frac 5{2 q}} \|\mathds{1}_{R_{\frac 3 2 t,x}} u\|_{\dot {\mathcal{M}}^{\frac{2q}5,q}_2}$.

We may now state our main result:

\begin{theorem}\label{maintheo}$\ $ Let $5<q<+\infty$; then:
\\ A) For    every $\sigma\in \mathfrak{S}_1$, the bilinear operator $B_\sigma$  is  a bounded bilinear operator from  $\mathcal{Y}_{KT,q}\times \dot{ \mathcal{M}}_2^{2,5}$ to  $\mathcal{Y}_{KT,q}$.
\\ B)   The following assertions are equivalent:
\begin{itemize}
\item[(i)] $e^{t\Delta}  u_0\in \mathcal{Y}_{2}$;
\item[(i)] $e^{t\Delta}  u_0\in \mathcal{Y}_{KT,q}$;
\item[(ii)]                 $ u_0\in {\rm BMO}^{-1}$.
\end{itemize} 
C)
 $ \mathcal{Y}_{KT,q}$  is an  adapted  Banach space.
Thus, there exists a positive constant $\epsilon_0$ such that, if $\|                                                                                                                       \vec u_0\|_{ {\rm BMO}^{-1}}+\| \mathcal{L}(\mathbb{F})\|_{\mathcal{Y}_{KT,q}}<\epsilon_0$, then the Navier--Stokes problem (\ref{oseen}) has a global mild solution $\vec u\in\mathcal{Y}_{KT,q}$.
\\ D) More generally, if  $\mathcal{Y}$ is an adapted space  such that $\mathcal{Y}\subset\dot{ \mathcal{M}}_2^{2,5}$, there exists  a positive constant $\epsilon_1$ such that, if  $\mathbb{F}=\mathbb{F}_1+\mathbb{F}_2$ and $\| \vec u_0\|_{ {\rm BMO}^{-1}}+\| \mathcal{L}(\mathbb{F}_1)\|_{\mathcal{Y}_{KT,q}}+ \| \mathcal{L}(\mathbb{F}_2)\|_{\mathcal{Y}}<\epsilon_1$, then the Navier--Stokes problem (\ref{oseen}) has a global mild solution $\vec u\in\mathcal{Y}_{KT,q}+\mathcal{Y}$.
\end{theorem}

\begin{proof} We only need to prove point A), i.e. to estimate $B_\sigma(u,v)$ in $L^2_{t,x}(Q_{T,x})$ and in $\dot {\mathcal{M}}^{\frac{2q}5,q}_2(R_{T,x})$ for $u\in \mathcal{Y}_{KT,q}$ and $v\in \dot{ \mathcal{M}}_2^{2,5}$.

In order to estimate $B_\sigma(u,v)$ in $L^2_{t,x}(Q_{T,x})$, we follow  the proof of Theorem \ref{kocht} given by Koch and Tataru and we   split $w=B_\sigma(u,v)$ in three parts: 
\begin{itemize}
\item[$\bullet$] $w_1=B_\sigma(u,(1-\mathds{1}_{Q_{10 T,x}} )v)$ : we saw that $$\|  \mathds{1}_{Q_{T,x}}  w_1\|_{L^2L^2}\leq C T^{3/4} \|u\|_{Y_2} \|v\|_{Y_2}.$$
\item[$\bullet$] $w_2(t,y)= \sigma(D) \sqrt{-\Delta}e^{t\Delta}\int_0^t \mathds{1}_{Q_{10 T,x}} uv\, ds$.  We saw that $$\|  \mathds{1}_{Q_{T,x} } w_2\|_{L^2L^2}\leq  \|w_2\|_{L^2 L^2} \leq C T^{3/4} \|u\|_{Y_2} \|v\|_{Y_2}.$$
\item[$\bullet$] $w_3(t,y)= \sigma(D) \int_0^t  (e^{(t-s)\Delta}-e^{t\Delta}) \sqrt{-\Delta}(\mathds{1}_{Q_{10 T,x}} uv)\, ds$.  We are going to prove below (Theorem \ref{theo5} in next section) that, more generally, 
$$\| \int_0^t  (e^{(t-s)\Delta}-e^{t\Delta}) \sqrt{-\Delta}(  uv)\, ds\|_{L^2L^2}\leq C  \sup_{T>0,x\in\mathbb{R}^3} T^{\frac 1 2-\frac 5{2q}} \|\mathds{1}_{R_{T,x}}   u\|_{ \dot {\mathcal{M}}^{\frac{2q}5,q}_2} \|v\|_{L^2L^2}.$$
Thus,
$$\|  \mathds{1}_{Q_{T,x}}  w_3\|_{L^2L^2}\leq  \|w_3\|_{L^2 L^2} \leq C \|u\|_{Y_{KT,q}}\|\mathds{1}_{Q_{10 T,x}} v\|_{L^2L^2}\leq C T^{3/4} \|u\|_{Y_{KT,q}} \|v\|_{Y_2}.$$
\end{itemize}
Hence, $w\in Y_2$.

In order to estimate $B_\sigma(u,v)$ in $\dot {\mathcal{M}}^{\frac{2q}5,q}_2$,
 we write $w=w_4+w_5$ with $w_4=B_\sigma((1-\mathds{1}_{S_{T,x}} ) u,v)$ and $w_5=B_\sigma(\mathds{1}_{S_{ T,x}}u,  v)$, where $$S_{T,x}=\{(t,y)\ /\ T/4<t<T, \vert x-y\vert\leq \sqrt {10 T}\}. $$  
We easily  check that $ \mathds{1}_{R_{ T,x}} \vert w_4\vert\leq C \frac 1 {\sqrt T} \|uv\|_{\dot {\mathcal {M}}^{1,5/2}_2}$ (see the proof of Proposition \ref{propmorrey}) and thus $$\|  \mathds{1}_{R_{T,x}}  w_4\|_{\dot {\mathcal{M}}^{\frac{2q}5,q}_2}\leq C\|  \mathds{1}_{R_{T,x}}  w_4\|_{L^qL^q}\leq C' T^{-\frac 1 2+\frac 5{2q}} \|u\|_{ \dot{ \mathcal{M}}_2^{2,5}} \|v\|_{ \dot{ \mathcal{M}}_2^{2,5}}.$$
On the other hand, we have
  $$      \vert w_5(t,z)\vert\leq C_\sigma  \int_0^t \int \frac 1{(\sqrt{t-s}+\vert z-y\vert)^4} \mathds{1}_{S_{ T,x}}(s,y) \vert u(s,y)\vert\, \vert v(s,y)\vert\, dy\, ds.
 $$    Thus, $w_5$ is controlled by the parabolic Riesz potential of $\vert \mathds{1}_{S_{ T,x}}uv\vert$;  as  $\vert \mathds{1}_{S_{ T,x}}uv\vert\in \dot{ \mathcal{M}}_2^{\frac 2 5 \frac{5q}{5+q} ,\frac{5q}{5+q}}$ (since  $\mathds{1}_{S_{ T,x}}u\in \dot {\mathcal{M}}^{\frac{2q}5,q}_2$ and $v\in \dot{ \mathcal{M}}_2^{2,5}$), we conclude by Hedberg's inequality for Riesz potentals and Morrey spaces that $w_5$ is controlled in $ \dot {\mathcal{M}}^{\frac{2q}5,q}_2$:
 $$ \|w_5\|_{ \dot {\mathcal{M}}^{\frac{2q}5,q}_2}\leq C  \|\mathds{1}_{S_{ T,x}}u\|_{ \dot {\mathcal{M}}^{\frac{2q}5,q}_2} \|v\|_{ \dot{ \mathcal{M}}_2^{2,5}}\leq C' T^{\frac 5{2q}-\frac 1 2} \|u\|_{\mathcal{Y}_{KT,q}}\|v\|_{ \dot{ \mathcal{M}}_2^{2,5}}.\qedhere $$
\end{proof}

\begin{corollary}$\ $\\ a)  There exists  a positive constant $\epsilon_0$ such that, if  $\mathbb{F}=\mathbb{F}_1+\mathbb{F}_2$ and $\| \vec u_0\|_{ {\rm BMO}^{-1}}+\| \mathcal{L}(\mathbb{F}_1)\|_{\mathcal{Y}_{KT,q}}+ \| \mathcal{L}(\mathbb{F}_2)\|_{\mathcal{M}(\dot H^{1/2,1}_{t,x}\mapsto L^2_{t,x})}<\epsilon_0$, then the Navier--Stokes problem (\ref{oseen}) has a global mild solution $\vec u\in\mathcal{Y}_{KT,q}+\mathcal{M}(\dot H^{1/2,1}_{t,x}\mapsto L^2_{t,x})$.
\\ b) For $5<p<+\infty$, there  exists  a positive constant $\epsilon_p$ such that, if  $\mathbb{F}=\mathbb{F}_1+\mathbb{F}_2+\mathbb{F}_3$ and $$\| \vec u_0\|_{ {\rm BMO}^{-1}}+\| \mathcal{L}(\mathbb{F}_1)\|_{\mathcal{Y}_{KT,q}}+ \| \mathcal{L}(\mathbb{F}_2)\|_{\dot{ \mathcal{M}}_2^{p,5}}+  \| \mathcal{L}(\mathbb{F}_3)\|_{ L^2  \mathrm{A} }<\epsilon_p,$$  then the Navier--Stokes problem (\ref{oseen}) has a global mild solution $\vec u\in\mathcal{Y}_{KT,q}+\dot{ \mathcal{M}}_2^{p,5}+L^2  \mathrm{A} $.
\end{corollary}

\begin{proof} a) is direct consequence of Theorem \ref{maintheo}, as $\mathcal{M}(\dot H^{1/2,1}_{t,x}\mapsto L^2_{t,x})$ is an adapted Banach space contained in $\dot{ \mathcal{M}}_2^{2,5}$ \cite{PGL2, PGL3, DAO}. Similarly, $\dot{ \mathcal{M}}_2^{p,5}$ is an adapted  Banach space contained in $\dot{ \mathcal{M}}_2^{2,5}$  \cite{PGL2}, $L^2  \mathrm{A}$  is an adapted  Banach space contained in $\dot{ \mathcal{M}}_2^{2,5}$ \cite{LEI}, and,  and, for every    $\sigma\in \mathfrak{S}_1$, the bilinear operator $B_\sigma$  is a bounded bilinear operator from  $\dot{ \mathcal{M}}_2^{p,5}\times L^2  \mathrm{A}$  to  $\dot{ \mathcal{M}}_2^{p,5}$, by Proposition \ref{propmorrey} since $L^2  \mathrm{A}\subset L^{2,\infty}L^\infty$.
\end{proof}

   \section{Parabolic dyadic decomposition of the time-space domain}

 We decompose $(0,+\infty)\times \mathbb{R}^3$ as 
 $$ (0,+\infty)\times \mathbb{R}^3=\bigcup_{j\in\mathbb{Z}, k\in\mathbb{Z}^3} \{(t,x)\ / \ 1\leq 4^jt<4, 2^jx-k\in [0,1)^3\}= \bigcup_{j\in\mathbb{Z}, k\in\mathbb{Z}^3} R_{j,k} $$
 and $(0, 16\, 4^{-j})\times\mathbb{R}^3$ as
 $$(0, 16\, 4^{-j})\times\mathbb{R}^3= \bigcup_ {k\in\mathbb{Z}^3}  \{(t,x)\ / 0< 4^jt<16, 2^jx-k\in [0,1)^3\}= \bigcup_{k\in\mathbb{Z}^3} Q_{j,k} .$$
 If $v\in L^2 L^2$, then $v$ can be decomposed in an orthogonal series $$v=\sum_{j\in\mathbb{Z}, k\in\mathbb{Z}^3} \mathds{1}_{R_{j,k}} v=\sum_{j\in\mathbb{Z}, k\in\mathbb{Z}^3} v_{j,k} $$ with $$\|v\|_{L^2 L^2}^2=\sum_{j\in\mathbb{Z}, k\in\mathbb{Z}^3} \|v_{j,k}\|_{L^2 L^2}^2.$$
 Similarly, if $u\in  {Y}_{KT,q}$, then 
 $$u=\sum_{j\in\mathbb{Z}, k\in\mathbb{Z}^3} \mathds{1}_{R_{j,k}} u	=\sum_{j\in\mathbb{Z}, k\in\mathbb{Z}^3} u_{j,k} $$ with, for every $5/2\leq \rho\leq q$,  $$ \sup_{j\in\mathbb{Z}, k\in\mathbb{Z}^3} 2^{j(1-\frac 5 \rho)} \|u_{j,k}\|_{ \dot {\mathcal{ M}}_2^{\frac 2 5 \rho,\rho}}<+\infty.$$
 
 \begin{theorem}\label{theo5}$\ $\\ Let $v\in L^2 L^2$ and $u\in Y_{KT,q}$ with $5<q<+\infty$. Write $v_{j,k} =\mathds{1}_{R_{j,k}} v$, $v_j=\sum_{k\in\mathbb{Z}^3} v_{j,k}$, and $u_{j,k} =\mathds{1}_{R_{j,k}} u$. Then
 \\ A) For $2<r\leq\rho\leq q$  (with $r\leq\frac 2 5 q$) and $\alpha= 1-\frac 5 \rho$, 
 $$\| \int_0^t  (e^{(t-s)\Delta}-e^{t\Delta}) \sqrt{-\Delta}^{1+\alpha}(  uv_j)\, ds\|_{L^2L^2}\leq C  \|v_j\|_2 \sup_{k\in\mathbb{Z}^3}   \| u_{j,k}\|_{ \dot {\mathcal{ M}}_2^{r,\rho}}.$$
 \\ B) We have $$\| \int_0^t  (e^{(t-s)\Delta}-e^{t\Delta}) \sqrt{-\Delta}(  uv)\, ds\|_{L^2L^2}\leq C   \|v\|_{L^2L^2}\sup_{T>0,x\in\mathbb{R}^3} T^{\frac 1 2-\frac 5{2q}} \|\mathds{1}_{R_{T,x}}   u\|_{ \dot {\mathcal{ M}}_2^{\frac 2 5 q,q}}.$$
 \end{theorem}
 
 \begin{proof} $\ $\\
 
 \noindent {\bf Proof of A).}\\We first consider $4^j t\leq 16$  and estimate $W= \int_0^t  (e^{(t-s)\Delta}-e^{t\Delta}) \sqrt{-\Delta}^{1+\alpha}(  uv_j)\, ds$ in $L^2((0, 16\, 4^{-j}),L^2)$, then estimate $W^*=  \int_0^{16\, 4^{-j}}  (e^{(16\, 4^{-j}-s)\Delta}-e^{16\, 4^{-j}\Delta})                                                                                                                                                                                                                   \sqrt{-\Delta}^{\alpha}(  uv_j)\, ds$  in $L^2(\mathbb{R}^3)$ and finally we estimate $W= \int_0^t  (e^{(t-s)\Delta}-e^{t\Delta}) \sqrt{-\Delta}^{1+\alpha}(  uv_j)\, ds$ in $L^2((16 \, 4^{-j},+\infty), L^2)$.
 
 When $t<16\, 4^{-j}$, we write
 $$ W=\sum_{l\in\mathbb{Z}^3} \mathds{1}_{Q_{j,l}}  \int_0^t  (e^{(t-s)\Delta}-e^{t\Delta}) \sqrt{-\Delta}^{1+\alpha}(\sum_{k\in\mathbb{Z}^3}  u_{j,k}v_{j,k})\, ds$$ which we reorganize as
 $$ W=\sum_{m\in\mathbb{Z}^3}  \sum_{k\in\mathbb{Z}^3} \mathds{1}_{Q_{j,k+m}}  \int_0^t  (e^{(t-s)\Delta}-e^{t\Delta}) \sqrt{-\Delta}^{1+\alpha} (u_{j,k}v_{j,k})\, ds=\sum_{m\in\mathbb{Z}^3} W_m.$$ We have
 \begin{equation*}\begin{split}
 \|W\|_{L^2((0,16\, 4^{-j}),L^2} \leq & \sum_{m\in\mathbb{Z}^3}   \|W_m\|_{L^2((0,16\, 4^{-j}),L^2} \\=\sum_{m\in\mathbb{Z}^3}  \left(  \sum_{k\in\mathbb{Z}^3} \right.&\left. \| \mathds{1}_{Q_{j,k+m}}  \int_0^t  (e^{(t-s)\Delta}-e^{t\Delta}) \sqrt{-\Delta}^{1+\alpha} (u_{j,k}v_{j,k})\, ds\|_{L^2((0,16\, 4^{-j}),L^2}^2\right)^{1/2}
 \end{split}\end{equation*}
 We have
 \begin{equation*}\begin{split} \vert \int_0^t  &(e^{(t-s)\Delta}-e^{t\Delta}) \sqrt{-\Delta}^{1+\alpha} (u_{j,k}v_{j,k})\, ds\vert \leq C Z_\alpha(u_{j,k} v_{j,k}). \end{split}\end{equation*} where
 $$ Z_\alpha(w)= \int_0^t\int \frac 1{(\sqrt{t-s}+\vert x-y\vert)^{4+\alpha}} \vert w(s,y)\vert\, dy\, ds.$$
 $Z_\alpha$ is a parabolic Riesz potential and we have the following equivalent of the Fefferman-Phong inequality \cite{FEF} for the parabolic Riesz potentials and the parabolic Morrey spaces \cite{PGL2}: if $0<\beta<\frac 52$ and $2<p< \frac 5 \beta$, then
 $$ \|\int_0^t \int \frac 1{(\sqrt{t-s}+\vert x-y\vert)^{5-\beta}} \vert f(s,y) g(s,y)\vert\, dy\, ds\|_{L^2L^2}\leq C_{p,\beta} \|f\|_{L^2L^2} \|g\|_{\dot{\mathcal{M}}_2^{p, \frac 5 \beta}}.
 $$
 As $4+\alpha=5-\frac 5 \rho$, we get
 $$\|  Z_\alpha(u_{j,k}v_{j,k})\|_{L^2_{t,x}}\leq C    \|u_{j,k}\|_{{ \dot {\mathcal{ M}}_2^{r,\rho}}} \|v_{j,k}\|_{L^2L^2}.$$
 Thus, 
 $$ \|W_m\|_{L^2 L^2}\leq C \sup_{k\in\mathbb{Z}^3} \|u_{j,k}\|_{{ \dot {\mathcal{ M}}_2^{r,\rho}}} \|v_j\|_{L^2L^2}.$$ Moreover, if $\vert m\vert\geq 20$ and $\vert m_0\vert=10$, we have, for $0\leq s\leq t \leq 16 \, 4^{-j}$, $y\in Q_{j,k}$, $x\in Q_{j,k+m}$ and $z\in Q_{j,k+m_0}$,
  \begin{equation*}\begin{split} \frac 1{(\sqrt{t-s}+\vert x-y\vert)^{4+\alpha}} \leq& \frac 1{ \vert x-y\vert^{4+\alpha}} \leq  C \frac{2^{(4+\alpha)j}}{m^{4+\alpha}}\leq C' \frac 1{m^{4+\alpha}  (\sqrt{t-s}+\vert z-y\vert)^{4+\alpha}}
  \end{split}\end{equation*} so that, for $0< t\leq 16\, 4^{-j}$, 
    \begin{equation*}\begin{split}
    \mathds{1}_{Q_{j,k+m}}(x)  Z_\alpha(u_{j,k}v_{j,k})&(t,x)\\\leq \frac C{m^{4+\alpha}}& \mathds{1}_{Q_{j,k+m_0}}(x-(m-m_0)2^{-j})  Z_\alpha(u_{j,k}v_{j,k})(t,x-(m-m_0)2^{-j}))
      \end{split}\end{equation*} and  $$ \|W_m\|_{L^2((0,16\, 4^{-j}), L^2)}\leq C \frac 1{m^{4+\alpha}} \sup_{k\in\mathbb{Z}^3} \|u_{j,k}\|_{{ \dot {\mathcal{ M}}_2^{r,\rho}}} \|v_j\|_{L^2L^2}.$$ Thus, we have proved that
      $$ \|W\|_{L^2((0,16\, 4^{-j}), L^2)}\leq C  \sup_{k\in\mathbb{Z}^3} \|u_{j,k}\|_{{ \dot {\mathcal{ M}}_2^{r,\rho}}} \|v_j\|_{L^2L^2}.$$ 
      
      We now estimate  $W^*=  \int_0^{16\, 4^{-j}}  (e^{(16\, 4^{-j}-s)\Delta}-e^{16\, 4^{-j}\Delta})                                                                                                                                                                                                                   \sqrt{-\Delta}^{\alpha}(  uv_j)\, ds$. First, as $v_j$ is supported in $4^{-j}<t<4 \ 4^{-j}$, we write
       \begin{equation*}\begin{split} W^*=&  \int_{4^{-j}}^{4\, 4^{-j}}  (e^{(16\, 4^{-j}-s)\Delta}-e^{16\, 4^{-j}\Delta})                                                                                                                                                                                                                   \sqrt{-\Delta}^{\alpha}(  uv_j)\, ds
       \\ =&   \int_{4^{-j}}^{4\, 4^{-j}} \int_0^s  (e^{(16\, 4^{-j}-s+\theta)\Delta} )                                                                                                                                                                                                                   \sqrt{-\Delta}^{2+\alpha}(  uv_j)\, d\theta\, ds
         \end{split}\end{equation*} so that
          \begin{equation*}\begin{split}  \vert W^*(x)\vert\leq& C\int_{4^{-j}}^{4\, 4^{-j}} \int_0^s\int  \frac 1{(\sqrt{16 \, 4^{-j}-s+\theta}+\vert x-y\vert)^{5+\alpha}}                                                                                                                                                                                                                  \vert   u(s,y)v_j(s,y)\vert\, dy\, d\theta\, ds
       \\ \leq& C'\int_{4^{-j}}^{4\, 4^{-j}} \int  \frac{4^{-j}}{(2^{-j}+\vert x-y\vert)^{5+\alpha}} \vert   u(s,y)v_j(s,y)\vert\, dy\, ds.
         \end{split}\end{equation*} 
         If $12\, 4^{-j}\leq \tau\leq 16\, 4^{-j}$, we have
         \begin{equation*}\begin{split} \int_{4^{-j}}^{4\, 4^{-j}} \int  &\frac{4^{-j}}{(2^{-j}+\vert x-y\vert)^{5+\alpha}} \vert   u(s,y)v_j(s,y)\vert\, dy\, ds
      \\   \leq  & C \int_0^\tau \int     \frac {2^{-j}}{(\sqrt{\tau-s}+\vert x-y\vert)^{4+\alpha}} \vert  u(s,y)v_j(s,y)\vert\, dy\, ds = C 2^{-j} Z_\alpha(\tau,x).      \end{split}\end{equation*} 
      Thus,
      $$ \|W^*(x)\|_2\leq C 2^{-j} \frac 1{4 \, 4^{-j}}\int_{12\, 4^{-j}}^{16 4^{-j}} \|Z_\alpha(\tau,.)\|_2\, d\tau\leq \frac C 2 \|Z_\alpha\|_{L^2((12\, 4^{-j},16\, 4^{-j}), L^2)}.$$ This gives
      $$ \|W^*\|_2\leq C  \sup_{k\in\mathbb{Z}^3} \|u_{j,k}\|_{{ \dot {\mathcal{ M}}_2^{r,\rho}}} \|v_j\|_{L^2L^2}.$$ 
      
      Finally, we   estimate $W= \int_0^t  (e^{(t-s)\Delta}-e^{t\Delta}) \sqrt{-\Delta}^{1+\alpha}(  uv_j)\, ds$ in $L^2((16 \, 4^{-j},+\infty), L^2)$. For $t> 16 \, 4^{-j}$, we have
      $$W= \int_0^{16\, 4^{-j}} (e^{(t-s)\Delta}-e^{t\Delta}) \sqrt{-\Delta}^{1+\alpha}(  uv_j)\, ds= \sqrt{-\Delta}e^{(t-16\, 4^{-j})\Delta} W^*$$
and thus
            $$ \|W\|_{L^2( (16 \, 4^{-j},+\infty), L^2)}\leq \frac 1{\sqrt 2}\|W^*\|_2\leq  C  \sup_{k\in\mathbb{Z}^3} \|u_{j,k}\|_{ \dot {\mathcal{ M}}_2^{r,\rho}} \|v_j\|_{L^2L^2}.$$ 

 $\ $\\
 
 \noindent {\bf Proof of B).}\\ Let $U= \int_0^t  (e^{(t-s)\Delta}-e^{t\Delta}) \sqrt{-\Delta}(  uv)\, ds$, $U=\sum_{j\in\mathbb{Z}} U_j$ with $$U_j= \int_0^t  (e^{(t-s)\Delta}-e^{t\Delta}) \sqrt{-\Delta}(  uv_j)\, ds.$$  Let $\gamma= 1-\frac 5 q$ and $\frac 1 \rho= \frac 2 5-\frac 1 q$. Then $\frac 5 2<\rho<5$ and $1-\frac 5 \rho=-\gamma$.  Let $2<r<\min(\frac{2q}5,\rho)$. From point A), we know that
 $$ \| U_j\|_{L^2 \dot H^{\gamma}}\leq    C  \|v_j\|_2 \sup_{k\in\mathbb{Z}^3}   \| u_{j,k}\|_{ \dot {\mathcal{ M}}_2^{r,q}} \leq C' 2^{j\gamma} \|v_j\|_2\sup_{T>0,x\in\mathbb{R}^3} T^{\frac 1 2-\frac 5{2q}} \|\mathds{1}_{R_{T,x}}   u\|_{ \dot {\mathcal{ M}}_2^{\frac 2 5 q,q}} .$$ and $$ \| U_j\|_{L^2 \dot H^{-\gamma}}\leq    C  \|v_j\|_2 \sup_{k\in\mathbb{Z}^3}   \| u_{j,k}\|_{ \dot {\mathcal{ M}}_2^{r,\rho}} \leq C' 2^{-j\gamma} \|v_j\|_2 \sup_{T>0,x\in\mathbb{R}^3} T^{\frac 1 2-\frac 5{2q}} \|\mathds{1}_{R_{T,x}}   u\|_{ \dot {\mathcal{ M}}_2^{\frac 2 5 q,q}} .$$ 
 We then have
      \begin{equation*}\begin{split} \int_0^{+\infty}\int \vert U(t,x)\vert^2\, dx\, dt=&\sum_{j\in\mathbb{Z}}   \int_0^{+\infty}\int \vert U_j(t,x)\vert^2\, dx\, dt\\& +2 \sum_{j\in\mathbb{Z}} \sum_{k\in\mathbb{Z},  k< j}  \int_0^{+\infty}     \langle (-\Delta)^{-\gamma} U_j(t,.) \vert (-\Delta)^\gamma U_k(t,.) \rangle \, dt
      \\ \leq    C  (\sup_{T>0,x\in\mathbb{R}^3}   T^{\frac 1 2-\frac 5{2q}} \|\mathds{1}_{R_{T,x}}   u\|_{ \dot {\mathcal{ M}}_2^{\frac 2 5 q,q}})^2
         &(\sum_{j\in\mathbb{Z}}  \|v_j\|_{L^2L^2}^2    +2  \sum_{j\in\mathbb{Z}} \sum_{k\in\mathbb{Z},  k< j} 2^{-\gamma(j-k)} \|v_j\|_{L^2 L^2}\|v_k\|_{L^2 L^2} )\\ \leq& C'  (\sup_{T>0,x\in\mathbb{R}^3}  T^{\frac 1 2-\frac 5{2q}} \|\mathds{1}_{R_{T,x}}   u\|_{ \dot {\mathcal{ M}}_2^{\frac 2 5 q,q}})^2
            \|v\|_{L^2L^2}^2.   \end{split} \end{equation*} 
The theorem is proved.
 \end{proof}


\section*{Appendix: counter-examples.}

\begin{proposition}\label{propy2}$\ $\\   Let $\sigma_0(\xi)=\vert \xi\vert$. Then $\sigma_0\in \mathfrak{S}_1$ and $B_{\sigma_0}$ is not bounded on $  \mathcal{Y}_2$. It is not bounded as well from $\mathcal{M}^{2,5}_2\times \mathcal{M}^{2,5}_2$ to   $\mathcal{Y}_2$.
 \end{proposition}

\begin{proof} Due to the invariance of the norms of $\mathcal{Y}_2$ a:nd $ \mathcal{M}^{2,5}_2$ through translations and dilations, the operator $B_{\sigma_0}$ would be bounded    from $\mathcal{M}^{2,5}_2\times \mathcal{M}^{2,5}_2$ to   $\mathcal{Y}_2$.
 if and only if there would exist a constant $C_0$ such that, for every $u,v\in  \mathcal{M}^{2,5}_2$,
$$ \int_0^1 \int_{[-1,1]^3} \vert  B_{\sigma_0}(u,v)\vert^2 \, dt\, dx\leq C_0 \|u\|_{ \mathcal{M}^{2,5}_2}^2  \|v\|_{ \mathcal{M}^{2,5}_2}^2.
$$ 
We then take $u_n(t,x)=v_n(t,x) = \psi_n(x_1,x_2)$ with $\psi_n\in L^2(\mathbb{R}^2)$. We have, for $T>0$,, $t_0>0$ and $x_0\in \mathbb{R}^3$,
$$ \int_{t_0}^{t_0+T }\int_{B(x_0,\sqrt T)} \vert u_n(t,x)\vert^2\, dt\, dx\leq   2 \|\psi_n\|_2^2 T^{3/2}.$$ Thus, $u_n\in  \mathcal{M}^{2,5}_2$. 

Moreover, for $w_n(x)=u_n(t,x)^2$ and $\phi_n=\psi_n^2$, we have the Fourier transforms
$$ \mathcal{F}(e^{(t-s)\Delta}\sigma_0(D)w_n)= e^{-(t-s)\vert\xi\vert^2} \vert\xi\vert (\hat {\phi_n}(\xi_1,\xi_2)\otimes\delta(\xi_3))$$ hence $$ \mathcal{F}(e^{(t-s)\Delta}\sigma_0(D)w_n)=   \left(e^{-(t-s)\vert(\xi_1,\xi_2)\vert^2} \vert(\xi_1,\xi_2)\vert \hat {\phi_n}(\xi_1,\xi_2) \right)\otimes\delta(\xi_3).$$ Hence, writing $\Delta_2$ for the laplacian operator on $\mathbb{R}^2$, we get 
 $$B_{\sigma_0}(u_n,u_n)=\int_0^t e^{(t-s)\Delta_2}\sqrt{-\Delta_2} \phi_n\, ds=  (\Id-e^{t\Delta_2}) \frac 1{\sqrt{-\Delta_2}}\phi_n.$$ From $\frac 1{\vert (y_1,y_2)\vert}\in L^1(\mathbb{R}^2)+L^\infty(\mathbb{R}^2)$, we find
$$ \vert e^{t\Delta_2} \frac 1{\sqrt{-\Delta_2}}\phi_n\vert \leq C \|\phi_n\|_1 (1+\frac 1 t)$$ and thus
$$ \int_{1/2}^1 \int_{[-1,1]^3}    \vert e^{t\Delta_2} \frac 1{\sqrt{-\Delta_2}}\phi_n  \vert^2 \, dt\, dx\leq 36 C^2 \|\phi_n\|_1^2.$$
In particular,
$$ \int_{1/2}^1 \int_{[-1,1]^3} \vert  \frac 1{\sqrt{-\Delta_2}}\phi_n\vert^2 \, dt\, dx\leq 8 (8 C_0+36 C^2)  \|\phi_n\|_1^2.
$$
Thus, if $\phi_n$ is an approximation of the Dirac mass with $\|\phi_n\|_1=1$, we find that the function $ \frac 1{\sqrt{-\Delta_2}}\phi_n$ is bounded in $L^2((-1,1)^2)$; but it converges in $\mathcal{D}'((-1,1)^2)$ to $\frac{\sqrt{2\pi}}{\vert y\vert} $ which is not square-integrable on $(-1,1)^2$. Thus, $B_{\sigma_0}$ is not bounded  from $\mathcal{M}^{2,5}_2\times \mathcal{M}^{2,5}_2$ to     $\mathcal{Y}_{2}$.
\end{proof}

\begin{proposition}\label{propy3}$\ $\\   Let $\sigma_0(\xi)=\vert \xi\vert$. Then $\sigma_0\in \mathfrak{S}_1$ and $B_{\sigma_0}$ is not bounded from $  \mathcal{Y}_{KT}\times L^2\mathrm{A}$ to $L^2\mathrm{A}$ nor to $  \mathcal{Y}_{KT}$ .\end{proposition}

\begin{proof}
Let $u(t,x)=\mathds{1}_{(0,2)}(t) \frac{x_1}{\vert x_1\vert}$ and $v(t,x)=\mathds{1}_{(0,2)}(t) \phi(x_1) \psi(x_2,x_3)$ where the Fourier transforms of $\phi$ and $\psi$ are integrable (over $\mathbb{R}$ and over $\mathbb{R}^2$ respectively) and the support of the Fourier transform of $\psi$ is contained in the corona $1<\xi_2^2+\xi_3^2<4$. Then $u\in \mathcal{Y}_{KT}$ and $v\in L^2\mathrm{A}$.

 For $t\in (0,2)$, we have
 $$ \mathcal{F}(B_{\sigma_0} (u,v))= 2\pi\int_0^t e^{-(t-s)\vert \xi\vert^2} \vert \xi\vert  H(\hat \phi)(\xi_1) \hat \psi(\xi_2,\xi_3)\, ds= 2\pi\frac{1-e^{-t\vert\xi\vert^2}}{\vert \xi\vert}  H(\hat \phi)(\xi_1) \hat \psi(\xi_2,\xi_3)$$ where $H$ is the Hilbert transform. For $1<t<2$, we have
 $$ \int \vert \mathcal{F}(B_{\sigma_0} (u,v))\, d\xi \vert \geq \pi (1-e^{-1})  \int_{-1}^1 \vert H(\hat \phi)(\xi_1) \vert \, d\xi_1 \iint \frac{\vert \hat \psi(\xi_2,\xi_3)\vert }{\vert (\xi_2,\xi_3)\vert}\, d\xi_2\, d\xi_3$$ and thus
 $$   \int_{-1}^1 \vert H(\hat \phi)(\xi_1) \vert \, d\xi_1 \leq C \|B_{\sigma_0} (u,v)\|_{L^2\mathrm{A}}$$ whereas
 $$ \|u\|_{\mathcal{Y}_{KT}} \|v\|_{L^2\mathrm{A}}\leq C \|\hat\phi\|_1.$$  Assuming that $B_{\sigma_0}$ is   bounded from $  \mathcal{Y}_{KT}\times L^2\mathrm{A}$ to $L^2\mathrm{A}$, we obtain that, for $\phi\in L^1(\mathbb{R})$,  $$   \int_{-1}^1 \vert H(\hat \phi)(\xi_1) \vert \, d\xi_1\leq C \|\hat\phi\|_1.$$ Applying this to $\phi_R=\phi(\frac x R)$, we have
  $$   \int_{-R}^R \vert H(\hat \phi)(\xi_1) \vert \, d\xi_1=     \int_{-1}^1 \vert H(\widehat{ \phi_R})(\xi_1) \vert \, d\xi_1\leq C \|\widehat{\phi_R}\|_1=C \|\hat\phi\|_1.$$ Letting $R$ go to $+\infty$, we find that the Hilbert transform is bounded on $L^1$, which is false.  Thus,  $B_{\sigma_0}$ cannot be bounded from $  \mathcal{Y}_{KT}\times L^2\mathrm{A}$ to $L^2\mathrm{A}$.

  Now, consider a function $\phi$ whose Fourier transform $\hat\phi\in\mathcal{D}(\mathbb{R}^3)$ is non-negative, supported in $B(0,1)$ and with $\|\hat\phi\|_1=1$. Similarly, consider a function $\psi$ whose Fourier transform $\hat\phi\in\mathcal{D}(\mathbb{R}^3)$ is non-negative, supported in $B(0,4)\setminus B(0,2)$ and with $\|\hat\psi\|_1=1$. Writing, for a function $f$ on $\mathbb{R}^3$ and for $\epsilon>0$, $f_\epsilon(x)= f(\frac x\epsilon)$, we define $$u(t,x)=\mathds{1}_{(0,1)}(t) \phi_{\sqrt{1-t}}(x) \text{  and } v(t,x)=\mathds{1}_{(0,1)}(t) \frac 1{\sqrt{1-t} \vert \ln (2-2t )\vert^{3/4}}  \psi_{\sqrt{1-t}}(x).$$    Let us assume that  $B_{\sigma_0}$ is bounded from $  \mathcal{Y}_{KT}\times L^2\mathrm{A}$ to $\mathcal{Y}_{KT}$. As $u\in \mathcal{Y}_{KT}$ and $v\in L^2\mathrm{A}$, we should have $\sup_{t>0}\sqrt t \|B_{\sigma_0}(u,v)(t,.)\|_\infty<+\infty$, where the supremum bound  holds for every $t>0$ by weak-* continuity of $t>0\mapsto B_{\sigma_0}(u,v)(t,.)\in L^\infty$. In particular, $B_{\sigma_0}(u,v)(1,.)\in L^\infty$. The Fourier transform $B_{\sigma_0}(u,v)(1,.)$ is non-negative, as
  $$\mathcal{F}(B_{\sigma_0}(u,v)(1,.)) = \int_0^1 e^{-(1-s)\vert\xi\vert^2} \vert \xi\vert \frac 1{\sqrt{1-s} \vert \ln (2-2s )\vert^{3/4}} \mathcal{F}(\phi_{\sqrt{1-s}}\psi_{\sqrt{1-s}})(\xi)\, ds,$$ and thus
  $$ \|B_{\sigma_0}(u,v)(1,.)\|_\infty=\frac 1{(2\pi)^3 } \int \mathcal{F}(B_{\sigma_0}(u,v)(1,.))(\xi)\, d\xi.$$ The support of $\mathcal{F}(\phi_{\sqrt{1-s}}\psi_{\sqrt{1-s}})(\xi)$ is contained in $\{\xi\ /\ \frac 1{\sqrt{1-s}}\leq \vert \xi\vert\leq \frac 5{\sqrt{1-s}}\}$
  and thus, for a constant $\gamma>0$, 
  $$ \int \mathcal{F}(B_{\sigma_0}(u,v)(1,.))(\xi)\, d\xi\geq \gamma \int_0^1\int \frac 1{(t-s)\vert \ln (2-2s )\vert^{3/4}} \mathcal{F}(\phi_{\sqrt{1-s}}\psi_{\sqrt{1-s}})(\xi)\, ds\, d\xi$$ and thus
\begin{equation*} \int \mathcal{F}(B_{\sigma_0}(u,v)(1,.))(\xi)\, d\xi\geq \gamma \|\hat\phi\|_1  \|\hat\psi\|_1  \int_0^1  \frac 1{(t-s)\vert \ln (2-2s )\vert^{3/4}}  \,  ds =+\infty.  \end{equation*} Thus,  $B_{\sigma_0}$ cannot be bounded from $  \mathcal{Y}_{KT}\times L^2\mathrm{A}$ to $\mathcal{Y}_{KT}$.
\end{proof}

 \begin{proposition}\label{propy4}$\ $\\    The space $L^2\mathrm{A}$ is not included in $\mathcal{M}(\dot H^{1/2,1}_{t,x}\mapsto L^2_{t,x})$. \end{proposition}
 \begin{proof}
If $L^2\mathrm{A} \subset\mathcal{M}(\dot H^{1/2,1}_{t,x}\mapsto L^2_{t,x})$, then the embedding would be continuous (by Baire's theorem, since convergence of a sequence  in $L^2\mathrm{A}$ implies convergence almost everywhere of a subsequence and since $\mathcal{M}(\dot H^{1/2,1}_{t,x}\mapsto L^2_{t,x})$ has the Fatou property that an almost everywhere converging sequence of bounded functions in $\mathcal{M}(\dot H^{1/2,1}_{t,x}\mapsto L^2_{t,x})$  remains in $\mathcal{M}(\dot H^{1/2,1}_{t,x}\mapsto L^2_{t,x})$  with the same bound).

Since the operator $T$ defined by $$T(u,v)(t,x)=\int_0^t \int \frac 1{(\sqrt{t-s}+\vert x-y\vert)^4} u(s,y)v(s,y)\, ds\, dy$$ is bounded on $\mathcal{M}(\dot H^{1/2,1}_{t,x}\mapsto L^2_{t,x})$  and since $\mathcal{M}(\dot H^{1/2,1}_{t,x}\mapsto L^2_{t,x})\subset \dot {\mathcal{M}}_2^{2,5}$, we would have the inequality
$$ \int_0^1\int_{B(0,1)} \vert T(u,u)(t,x)\vert\, dt\, dx\leq C \|u\|_{L^2\mathrm{A}}^2.$$
Let $u_R(t,x)=\omega(t) e^{-\frac{\vert x\vert^2}R}$. By monotonous convergence, we have pointwise  convergence of $T(u_R,u_R)(t,x)$ to 
$$\int_0^t \int \frac 1{(\sqrt{t-s}+\vert x-y\vert)^4} \omega^2(s), ds\, dy= C \int_0^t \frac 1{\sqrt{t-s}} \omega^2(s)\, ds.$$ Thus, we would have
$$\int_0^1 \left( \int_0^t \frac 1{\sqrt{t-s}} \omega^2(s)\, ds\right)^2\, dt \leq C \|\omega\|_2^4.$$
For $\omega_\epsilon(t)=\frac 1{\sqrt{\epsilon}} \theta(\frac x\epsilon)$ with $\theta\in\mathcal{D}$ and $\|\theta\|_2=1$, we have pointwise convergence (when $\epsilon\rightarrow 0$) of $\int_0^t \frac 1{\sqrt{t-s}} \omega_\epsilon^2(s)\, ds$ to $\frac 1{\sqrt t}$; Fatou's lemma would give $\int_0^1 \frac{dt}t\leq C$, which is false.
\end{proof}


\begin{thebibliography}{HD}



  \bibitem[Can99]{CPL}
 M. Cannone and F. Planchon, \emph {On the nonstationary Navier–Stokes equations with an external force},  Adv. Differential Equations {4} (1999), 697--730.

\bibitem[Dao17]{DAO}
N.A. Dao and Q.-H. Nguyen, \emph{Nonstationary Navier--Stokes equations with singular
time-dependent external forces}, C. R. Acad. Sci. Paris, Ser. I {355} (2017), 966--972.
 
 \bibitem[Fab72]{FJR}
E. Fabes, B.F. Jones  and N. Rivi\`ere, \emph {The initial value problem for the {N}avier---{S}tokes
equations with data in ${L}^p$},  Arch. Ration. Mech. Anal. {45} (1972), 222--240.

 
 \bibitem[Far16]{FGS}
 R.  Farwig, Y. Giga, and P.-Y. Shu, \emph{Initial values for the Navier–Stokes equations in spaces with weights in time},   Funkcial. Ekvac. {59} (2016), 199--216.

  

 \bibitem[Fef83]{FEF}
 C. Fefferman, 
\emph{The uncertainty principle}, Bull. Amer. Math. Soc. {9} (1983),  129--206.

\bibitem[Fer21]{FLR}
 P.G. Fern\'andez-Dalgo and P.G. Lemari\'e-Rieusset, 
\emph{Characterisation of the
pressure term in the incompressible {N}avier--{S}tokes equations on the whole space}, Discrete \& Continuous Dynamical Systems - S {14} (2021),  2917--2931.


 \bibitem[Koc01]{KOT}
H. Koch  and D. Tataru, 
\emph{Well-posedness for the {N}avier--{S}tokes equations}, Adv.
Math.  {157} (2001),  22--35.



 
  \bibitem[Koz96]{KOZ}
H. Kozono  and M. Nakao, 
\emph{Periodic solutions of the Navier--Stokes equations in unbounded domains}, Tohoku Math. J., {48} (1996), 33--50.


	   \bibitem[Koz18]{KOS}
  H. Kozono and S. Shimizu \emph{Navier–Stokes equations with external forces in Lorentz spaces and its application to the self-similar solutions}, J. Math. Anal. Appl. {458}  (2018), 1693--1708.
  
	   \bibitem[Koz94]{KOY}
  H. Kozono and M. Yamazaki, \emph{Semilinear heat equations and the
Navier--Stokes equations with distributions in new
function spaces as initial data}, Comm. Partial Differential Equations {19}  (1994), 959--1014.


  \bibitem[Lei11]{LEI}
Z. Lei  and F. Lin, 
\emph{Global mild  solutions of   Navier--Stokes equations}, Comm. Pure Appl. Math. {64} (2011), 1297--1304.

 


 \bibitem[Lem02]{PGL1}
P.G. Lemari\'e-Rieusset, 
\emph {Recent developments in the Navier--Stokes problem}, CRC Press, 2002.
 

 \bibitem[Lem16]{PGL2}
P.G. Lemari\'e-Rieusset, 
\emph {The Navier--Stokes Problem in the 21st Century}, CRC Press, 2016.

  

 \bibitem[Lem18]{PGL3}
P.G. Lemari\'e-Rieusset, 
\emph {Sobolev multipliers, maximal functions and parabolic equations with a quadratic nonlinearity}, Journal of Functional Analysis {274} (2018), 659--694




 \bibitem[Mey99]{MEY}
Yves Meyer,
\emph { Wavelets, paraproducts and  Navier--Stokes equations}, Current
developments in mathematics 1996, International Press, PO Box
38-2872, Cambridge, MA 02238-2872, 1999.



\end{thebibliography}
\end{document}